\documentclass{amsart}

\let\oldmarginpar\marginpar
\renewcommand\marginpar[1]{\-\oldmarginpar[\raggedleft\footnotesize #1]%
	{\raggedright\footnotesize #1}}

\usepackage{lineno}
\usepackage{mathptmx}
\usepackage{amsmath}
\usepackage{amscd}
\usepackage{amssymb}
\usepackage{amsthm}
\usepackage{xspace}
\usepackage[all,tips]{xy}
\usepackage{verbatim}
\usepackage{syntonly}
\usepackage{hyperref}
\usepackage{arydshln}
\usepackage{lineno,color}
\usepackage[capitalize]{cleveref}
\usepackage{relsize}
\usepackage{tikz}
\usepackage{float}

\theoremstyle{plain}
\newtheorem{thm}{Theorem}[section]
\newtheorem{theorem}[thm]{Theorem}
\newtheorem{cor}[thm]{Corollary}
\newtheorem{corollary}[thm]{Corollary}

\newtheorem{proposition}[thm]{Proposition}
\newtheorem{lemma}[thm]{Lemma}
\newtheorem{ques}[thm]{Question}

\theoremstyle{definition}

\newtheorem{remark}[thm]{Remark}
\newtheorem{defn}[thm]{Definition}
\newtheorem{definition}[thm]{Definition}

\DeclareMathOperator{\GL}{GL}

\newcommand{\homfill}{\operatorname{Fill}}
\newcommand{\fvol}{\operatorname{FVol}}

\newcommand{\R}{\ensuremath{\mathbb{R}}}
\newcommand{\Z}{\ensuremath{\mathbb{Z}}}

\usepackage{fancyhdr}

\newtheorem{innercustomthm}{}
\newenvironment{customthm}[1]
{\renewcommand\theinnercustomthm{#1}\innercustomthm}
{\endinnercustomthm}

\title{A coarse embedding theorem for homological filling functions}

\author{Robert Kropholler}

\author{Mark Pengitore}

\begin{document}

\begin{abstract}
We demonstrate under appropriate finiteness conditions that a coarse embedding induces an inequality of homological Dehn functions. Applications of the main results include a characterization of what finitely presentable groups may admit a coarse embedding into a hyperbolic group of geometric dimension $2$, characterizations of finitely presentable subgroups of groups with quadratic Dehn function with geometric dimension $2$, and to coarse embeddings of nilpotent groups into other nilpotent groups of the same growth and into hyperbolic groups.
\end{abstract}

\maketitle

\section{Introduction}

A coarse embedding of one metric space into another generalizes the notion of a quasi-isometry. It allows for the metric to be distorted by any functions that tend to infinity, rather than just linear functions allowed in quasi-isometries. 

\begin{definition}
	We say that a map $f\colon X\to Y$ is a {\em coarse embedding} if there are functions $\rho_-, \rho_+\colon \R\to \R$ such that $\lim_{x\to \infty}\rho_-(x) = \lim_{x\to \infty}\rho_+(x) = \infty$ and where the following inequality holds: 
	$$\rho_-(d(x, y)) \leq d(f(x), f(y)) \leq \rho_+(d(x, y)).$$
\end{definition}

In geometric group theory, a coarse embedding is a geometric version of subgroup containment. Indeed, if $H$ is a finitely generated subgroup of $G$, then the inclusion map is a coarse embedding. Thus, there are questions about which results on subgroups remain true when passing to coarse embeddings. 

In general, coarse embedding are wilder than subgroups. For instance, every finitely generated, infinite group admits a coarse embedding of $\Z$, however, there are infinite finitely generated torsion groups. 
Moreover, in the class of hyperbolic groups, there are coarse embeddings of $\Z^2$ into hyperbolic groups while it certainly can't be a subgroup.

For the statement of our theorems, we let $\homfill_{G}^{k}(\ell)$ be the $k$-dimensional homological Dehn function for $G$ i.e. the difficulty of filling $(k-1)$-cycles with $k$-chains. Additionally, if $f, g \colon \mathbb{N} \to \mathbb{N}$ are functions, we say that $f \prec g$ if there is a constant $C>0$ such that for all $n$ we have that $f(n) \leq C g(C n + C) + C n + C$. We say that $f \sim g$ when $f \prec g$ and $g \prec f$. For a group $G$, we denote $K(G,1)$ as a choice of an Eilenberg-Maclane space for $G$. We say a group $G$ has {\em geometric dimension $2$} if the minimal dimension of a $K(G,1)$ is $2$. Finally, we say that a group $G$ is of type $F_n$ if it admits a $K(G,1)$ with a compact $n$-skeleton.

In this paper, we study the preservation of homological filling functions under coarse embeddings. It is known that there are groups $H\subset G$ for which the homological filling function vary greatly. For instance, consider the inclusion $\GL_3(\Z) \hookrightarrow \GL_5(\Z)$, the former has exponential Dehn function \cite[Chapter 11]{epstein} whereas the latter has quadratic Dehn function \cite{younglinear}.

This type of example does not exist if $G$ is assumed to have cohomological dimension 2. See the remark after Theorem 4.6 of \cite{gersten}.
\begin{thm}
	Let $G$ be a group a cohomological dimension 2, and let $H$ be a finitely presented subgroup. 
	Then $\homfill_{H}^2(\ell)\prec\homfill_{G}^2(\ell).$
\end{thm}

In this paper, we take a more geometric approach and prove a similar theorem for coarse embeddings. 
As we cannot use the algebra of cohomology, we must replace the cohomological dimension assumption with the corresponding geometric dimension assumption. 
We prove the following. 

\begin{customthm}{\cref{main_thm1}}
	Let $G$ be a finitely presented group of geometric dimension $2$, and suppose that $H$ is a finitely presented group which admits a coarse embedding into $G$. Then $\homfill_{H}^2(\ell) \prec \homfill_{G}^2(\ell)$. 
\end{customthm}
Due to \cite[Theorem 5.2]{gersten}, we know that a finitely presented group $G$ is word hyperbolic if and only if $\homfill_{G}^2(\ell) \sim \ell$. Thus, the first consequence of the above theorem is a geometric generalization of Gersten's characterization of finitely presented subgroups of hyperbolic groups of geometric dimension $2$. \cite[Theorem 5.4]{gersten}.
\begin{customthm}{\cref{main_thm2}}
	Let $G$ be a hyperbolic group of geometric dimension $2$, and suppose $H$ is a finitely presented group that admits a coarse embedding into $G$. Then $H$ is a hyperbolic group. 
\end{customthm}

For the next corollary, we collect a few observations. First note that if $\delta_G(\ell)$ is the Dehn function of $G$, then $\homfill^2_G(\ell) \prec \delta_G(\ell)$. Moreover, \cite[Proposition 8]{mineyev} implies that if $\homfill^2_G(\ell)$ grows strictly slower than $\ell^2$, then $G$ is hyperbolic. Thus, we may use Theorem \ref{main_thm2} to obtain a characterization of subgroups of groups with quadratic Dehn function of geometric dimension $2$.
\begin{cor}
	Let $G$ be be a finitely presented group of geometric dimension $2$ with quadratic Dehn function. If $H$ is a finitely presented group that admits a coarse embedding into $G$, then either $H$ is a hyperbolic group or $H$ has quadratic Dehn function. In particular, all finitely presented subgroups of $G$ are either hyperbolic or have quadratic Dehn function.
\end{cor}

The next theorem generalizes \cite[Theorem 1.1]{hanlon_pedroza} to the context of coarse embeddings. It says that under appropriate finiteness conditions the higher dimensional homological Dehn filling functions are non-decreasing under coarse embeddings.
\begin{customthm}{\cref{main_thm3}}
Let $k \geq 1$. Suppose that $G$ admits a finite $(k+1)$-dimensional $K(G,1)$ and that $H$ is a finitely presented group of type $F_{k+1}$. If $H$ admits a coarse embedding into $G$, then $\homfill_H^{k+1}(\ell) \prec \homfill_G^{k+1}(\ell)$. 
\end{customthm}

Since the inclusion of a horosphere in hyperbolic $(k+1)$-space gives a coarse embedding of $\Z^k$ into $\pi_1(M)$ when $M$ is a closed $(k+1)$-dimensional real hyperbolic manifold, the above theorem cannot be improved. In contrast, the following application of Theorem \ref{main_thm3} demonstrates that there exists no coarse embedding of a torsion free, finitely generated nilpotent group whose integral cohomological dimension is $k+1$ into $\pi_1(M)$. Thus, whenever a torsion free, finitely generated nilpotent admits a coarse embedding into $\pi_1(M)$, its integral cohomological dimension must be strictly less than that of the integral cohomological dimension of $M$. This follows from the fact that \cite[Theorem 4]{lang} implies all homological filling functions for hyperbolic groups are linear and that \cite[IV 5.8 Theorem]{varopoulos_saloff-coste_coulhon} and \cite{gruber} together imply the top dimensional homological filling functions for torsion free, finitely generated nilpotent group are superlinear. Similarly, \cite[Theorem 3.3]{kropholler_kielak} implies that torsion free, finitely generated nilpotent groups of Hirsch length $k+1$ cannot coarsely embed into a nonamenable integral Poincar\'{e} duality group of dimension $k+1$.
\begin{cor}
Let $G$ be either a hyperbolic group that admits a finite $(k+1)$-dimensional $K(G,1)$ or a nonamenable integral Poincar\'{e} duality group of dimension $k+1$. Then a torsion-free, finitely generated nilpotent group of Hirsch length $k+1$ cannot admit a coarse embedding into $G$. 
\end{cor}

More precisely, \cite[IV 5.8 Theorem]{varopoulos_saloff-coste_coulhon} and \cite{gruber} imply that if $G$ is a torsion free, finitely generated nilpotent group of dimension $n$ and has degree of growth $d$, then $\homfill_{G}^{n}(\ell) \sim \ell^{\frac{d}{d-1}}$. Thus, an application of Theorem \ref{main_thm3} implies that $G$ and $H$ are torsion free, finitely generated nilpotent groups of the same integral cohomological dimension where $G$ has a degree of growth strictly less than that of $H$, then $G$ does not admit a coarse embedding into $H$. Since word growth is nondecreasing under coarse embeddings, we have that $H$ does not coarsely embed into $G$. For torsion free, finitely generated nilpotent groups of the same growth and integral cohomological dimension but having distinct asymptotic cones, \cite[Theorem 1.2]{Cohen_Pengitore} implies that there exists no coarse embedding from either $G$ to $H$ or vice versa. Thus, these facts together imply a nearly complete classification of what torsion free, finitely generated nilpotent groups of the same integral cohomological dimension admit coarse embeddings into each other. 
\begin{cor}
Let $G$ and $H$ be torsion free, finitely generated nilpotent groups of the same Hirsch length. If the asymptotic cones of $G$ and $H$ are not isomorphic as Lie groups, then $G$ and $H$ cannot admit coarse embeddings into one another. 
\end{cor}
The only case that remains is when two torsion free, finitely generated nilpotent groups $G$ and $H$ who have the same asymptotic cone.  These are groups will share many properties such as growth rate and integral cohomological dimension. Moreover, by  \cite[IV 5.8 Theorem]{varopoulos_saloff-coste_coulhon}, we have that $\homfill_G^{n}(\ell) \sim \homfill_H^{n}(\ell)$ where $n$ is the topological dimension of the asymptotic cone. In certain cases, such as when $G$ and $H$ are cocompact lattices in the associated asymptotic cone, we have that $G$ and $H$ are biLipschitz. Hence, they naturally admit coarse embeddings into each other. However, when either $G$ or $H$ is not a lattice in the asymptotic cone, then whether a coarse embedding or even a Lipschitz embedding exists is open.

Finally, we apply the techniques of Theorem \ref{main_thm1}, to obtain characterizations of finitely presented subgroups of infinitely presented small cancellation groups. 
\begin{customthm}{\cref{thm:smallcancellation}}
	Let $G$ be finitely generated group that admits an infinite $C'(1/6)$-small cancellation presentation where no relator is a proper power. 
	If $H$ is finitely presented group that admits a coarse embedding into $G$, then $H$ is a hyperbolic group. 
	In particular, all finitely presented subgroups of $G$ are hyperbolic.
\end{customthm}

Thus, small cancellation groups contain no finitely presented subgroup obstructions to hyperbolicity. However, there are many examples of small cancellation groups that cannot coarsely embed into hyperbolic groups \cite{hume_sisto}.

\subsection{Structure of the article}
In Section \ref{section:embedding}, we introduce basic definitions and embedding results between CW complexes associated to Lipschitz embeddings between finitely generated groups. Using these embedding results, Section \ref{section:filling_functions} introduces homological filling functions and relates homological fillings of codimension $0$ subcomplexes with fillings in the ambient complex. Section \ref{section:main_results} and Section \ref{small_cancellation} give the proofs of the main results.  Section \ref{section:future} finishes with some questions.

\noindent{\bf Acknowledgements:} 
The authors thank Ian Leary for answering questions concerning groups of geometric dimension $2$.
\section{Embedding lemmas}\label{section:embedding}

Throughout this article, let $f\colon H\to G$ be a Lipschitz embedding. 
In this section, we will obtain various embedding results between classifying spaces. 
Later, we will use these to understand filling functions, but first, we introduce some notation. When given a CW complex $X$, we denote its $k$-skeleta as $X^{(k)}$. We denote $S^k$ as the $k$-sphere and $D^k$ as the $k$-disk. 

\begin{definition}
	A {\em Cayley $n$-complex} for a group $G$ is an $n$-dimensional cell complex $X$ equipped with a free, cellular $G$-action with one orbit of vertices such that $\pi_i(X) = 1$ for all $i<n$. We say that $X$ is \emph{cocompact} if there are finitely many orbits of $k$-cells for each $k$.
\end{definition}

The 1-skeleton of a Cayley $n$-complex is a Cayley graph for $G$. 
Throughout we will assume that the $1$-skeleton is a simplicial graph. 
Thus, we may consider the map $f$ as a map of vertices from any Cayley complex for $G$ to any Cayley complex for $H$. 

The following lemma follows immediately from the definition of type $F_n$. 
\begin{lemma}
	A group $G$ is of type $F_n$ if and only if it admits a cocompact Cayley $n$-complex. 
\end{lemma}

This next proposition demonstrates that under appropriate finiteness conditions we may extend the Lipschitz map $f \colon H \to G$ to be an inclusion of cocompact Cayley $n$-complexes.

\begin{proposition}\label{lem:injection}
	Let $H$ be a group of type $F_n$ for $n>1$ and let $X$ be a cocompact Cayley $n$-complex for $H$. 
	
	Let $G$ be a group of type $F_{n-1}$, and let $Z$ be Cayley $n$-complex for $G$ with a cocompact $(n-1)$-skeleton
	
	Let $f\colon H\to G$ be a $C$-Lipschitz injection. 
	Then there exists a Cayley $n$-complex $Y$ for $H$ such that the following hold:
	\begin{itemize}
		\item $Y$ has cocompact $(n-1)$-skeleton, 
		\item $Z$ is a deformation retract of $Y$, 
		\item there is a map $\phi\colon X\to Y$ extending $f$, which is injective on $n-1$-skeleta, 
		\item there exists $N\geq 0$ such that each $n$-cell of $X$ is mapped to a collection of $\leq N$ $n$-cells in $Y$.
	\end{itemize}
\end{proposition}
\begin{proof}
	As mentioned above, we can consider $f$ as an injection at the level of $0$-skeleta. We now show how to extend this map across higher dimensional cells. 
	Since $f$ is $C$-Lipschitz, the end points of each edge of $X$ are mapped to points at distance at most $C$ apart. 
	Let $Z_1'$ be the space obtained from $Z$ by adding an edge between every pair of points $z_1, z_2$ such that $d(z_1, z_2)\leq C$. 
	Since $Z$ has a cocompact $1$-skeleton, we only add finitely many orbits of edges in this process. 	
	
	For each orbit of edges $[e]$ added to $Z$, let $\gamma_e$ be a minimal length path in $Z$ between the endpoints of $e$.  
	Now attach a 2-cell to $Z_1'$ with boundary $\gamma_e \bar{e}$. 
	Let $Z_1$ be the complex obtained by adding these disks equivariantly. 
	Since we only added finitely many orbits of edges to $Z$ to obtain $Z_1'$, we only add finitely many orbits of 2-cells to $Z_1'$ to obtain $Z_1$.
	We can now extend $f$ to an injective cellular map $X^{(1)}\to Z_1^{(1)}$. 
	See \cref{fig:extraedges} for a depiction of this procedure.
	
	\begin{figure}[H]
		\centering
		\begin{tikzpicture}[scale=0.5]
		\draw (0, 0) -- (1, -2) -- (3, -1) -- (4, -3) -- (5, -2) -- (6, -1) -- (7, -2) -- (8, 0);
		\fill (0,0) circle [radius=0.05] node [left] {$z_1$}; 
		\fill (8,0) circle [radius=0.05] node [right] {$z_2$}; 
		
		\filldraw [lightgray] (12, 0) -- (13, -2) -- (15, -1) -- (16, -3) -- (17, -2) -- (18, -1) -- (19, -2) -- (20, 0) -- (12, 0) ;
		\draw (12, 0) -- (13, -2) -- (15, -1) -- (16, -3) -- (17, -2) -- (18, -1) -- (19, -2) -- (20, 0);
		\draw [very thick] (20, 0) -- (12, 0);
		\fill (12,0) circle [radius=0.05] node [left] {$z_1$}; 
		\fill (20,0) circle [radius=0.05] node [right] {$z_2$}; 
		\draw (10, -1) node {$\leadsto$};
 		\end{tikzpicture}
		\caption{An extra edge is added between every pair of points $z_1, z_2$ at distance $\leq C$. The resulting copy of $S^1$ is then filled with a disk. }
		\label{fig:extraedges}
	\end{figure}
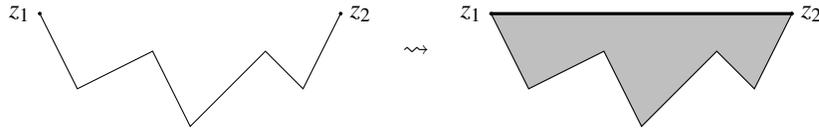
	
	We now proceed by induction.  
	Suppose that we have built a complex $Z_k$ which has a cellular embedding that extends $f$ to the $k$-skeleton of $X$ for some $k<n-1$ and where $Z$ is a deformation retract of $Z_k$.
	Moreover, assume $G$ acts cocompactly on the $(n-1)$-skeleton of $Z_k$. 
	Since $X$ is a cocompact Cayley $n$-complex, there are finitely many orbits of $(k+1)$-cells. Let $D_1, \dots, D_l$ be representatives of these orbits. Let $L$ be the maximum number of cells in the boundary of $D_i$. 
	
	Let $\mathcal{S}$ be the collection of cellular maps $S^k\to Z_k$ such that the image contains $\leq L$ $k$-cells. Let $Z_{k+1}'$ be the complex obtained by attaching a $k+1$-cell to each such map. Since $H$ acts cocompactly on the $(k+1) $-skeleton of $Z_k$, we only attach finitely many orbits of cells to obtain $Z_{k+1}'$. Thus, $H$ still acts cocompactly on the $(k+1)$-skeleton of $Z_{k+1}'$, and moreover, we have a cellular embedding $X^{(k+1)}\to Z_{k+1}'. $
	
	Now for each $(k+1)$-cell $c$ added to $Z_{k}$, its boundary is null-homotopic in $Z_k$. Therefore, it bounds a $(k+1)$-cell $d$ in $Z_k$. We can now attach a $(k+2)$-cell to $Z_{k+1}'$ with boundary which consists of $c$ on the upper hemisphere and $d$ on the lower hemisphere. Let the complex obtained this way be $Z_{k+1}$. 
	
	See \cref{fig:extradisks}, for a depiction of the 2-dimensional case.
	\begin{figure}[H]
		\centering
		\def\svgwidth{120mm}
\begingroup%
  \makeatletter%
  \providecommand\color[2][]{%
    \errmessage{(Inkscape) Color is used for the text in Inkscape, but the package 'color.sty' is not loaded}%
    \renewcommand\color[2][]{}%
  }%
  \providecommand\transparent[1]{%
    \errmessage{(Inkscape) Transparency is used (non-zero) for the text in Inkscape, but the package 'transparent.sty' is not loaded}%
    \renewcommand\transparent[1]{}%
  }%
  \providecommand\rotatebox[2]{#2}%
  \newcommand*\fsize{\dimexpr\f@size pt\relax}%
  \newcommand*\lineheight[1]{\fontsize{\fsize}{#1\fsize}\selectfont}%
  \ifx\svgwidth\undefined%
    \setlength{\unitlength}{447.63742331bp}%
    \ifx\svgscale\undefined%
      \relax%
    \else%
      \setlength{\unitlength}{\unitlength * \real{\svgscale}}%
    \fi%
  \else%
    \setlength{\unitlength}{\svgwidth}%
  \fi%
  \global\let\svgwidth\undefined%
  \global\let\svgscale\undefined%
  \makeatother%
  \begin{picture}(1,0.29788627)%
    \lineheight{1}%
    \setlength\tabcolsep{0pt}%
    \put(0,0){\includegraphics[width=\unitlength,page=1]{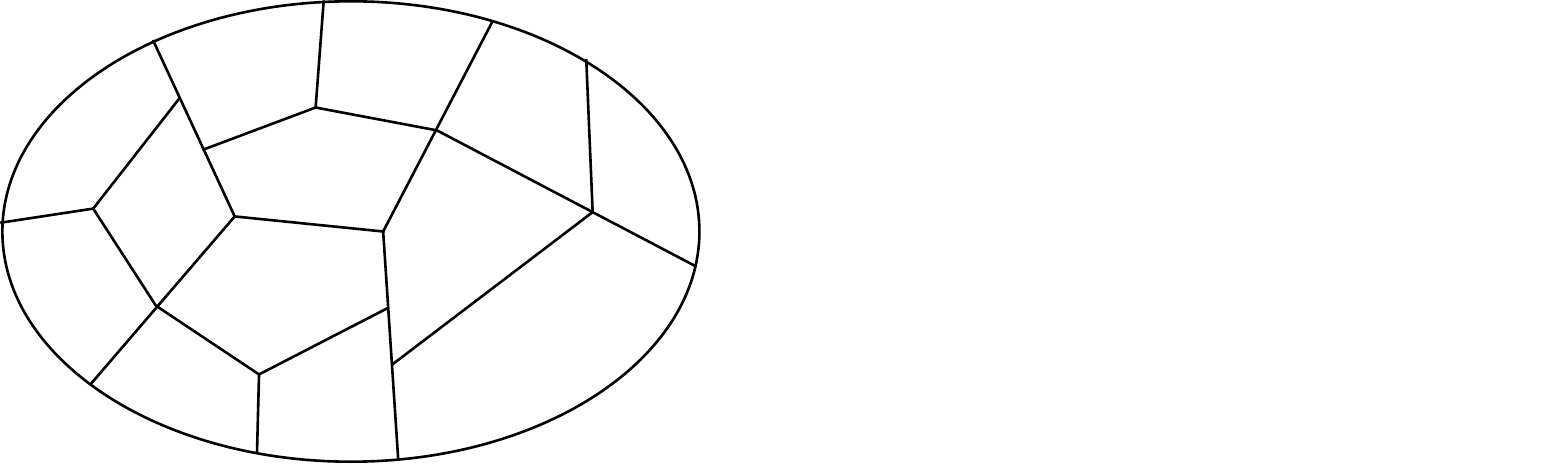}}%
    \put(0.40762267,0.03580227){\color[rgb]{0,0,0}\makebox(0,0)[lt]{\lineheight{1.25}\smash{\begin{tabular}[t]{l}$w$\end{tabular}}}}%
    \put(0,0){\includegraphics[width=\unitlength,page=2]{DD1.pdf}}%
    \put(0.48536588,0.14371134){\color[rgb]{0,0,0}\makebox(0,0)[lt]{\lineheight{1.25}\smash{\begin{tabular}[t]{l}$\leadsto$\end{tabular}}}}%
    \put(0,0){\includegraphics[width=\unitlength,page=3]{DD1.pdf}}%
  \end{picture}%
\endgroup%

		\caption{The 2-dimensional case from the proof of \cref{lem:injection}. We add an extra 2-cell for each disk diagram with boundary $w$, where $|w| \leq L$. We then fill the resulting sphere with a disk.}
		\label{fig:extradisks}
	\end{figure} 
	
	We will now show that $Z_{k}$ is a deformation retract of $Z_{k+1}$.
	Let $T = Z_{k+1}/Z_k$ be the space obtained by collapsing $Z_k$ to a point. 
	We can understand the space $T$ by building it as follows. 
	Let $S = Z_{k+1}'/Z_k$. 
	Then $S$ has the homotopy type of a wedge of $k$-spheres as we obtained $Z_{k+1}$ from $Z$ by attaching $k$-cells. 
	For each sphere in $S$ we attached a $(k+1)$-cell to $Z_{k+1}'$.
	After collapsing $Z_k$, the attaching map is homotopic to the identity $S^k\to S^k$.  
	Thus, we obtain $T$ by attaching $(k+1)$-cells to $S$ one for each sphere in $S$. 
	Hence, $T$ has the homotopy type of a wedge of copies of $D^{k+1}$ which is contractible. 
	Since $Z_{k}$ is a cellular subcomplex of $Z_{k+1}$ with contractible quotient, we see that $Z_k$ is a deformation retract of $Z_{k+1}$.
	Thus, by induction $Z$ is a deformation retract of $Z_{k+1}$. 
	
	Now assume that we have a cellular embedding at the level of $(n-1)$-skeleta. 
	By assumption, there are finitely many orbits of $n$-cells in $X$. 
	Let $\mu_{i, h}\colon S^{n-1}\to X$ be the attaching map of the $i$-th $n$-cell based at $h$. 
	The number of cells in the image of $\mu_{i, h}$ is uniformly bounded as $S^{n-1}$ is compact and there are only finitely many orbits of $n$-cells. 
	Let $L$ be the bound on the number of $(n-1)$-cells in $\mu_{i, h}(S^{n-1})$. 
	We have a cellular embedding at the level of $(n-1)$-skeleta; thus, $\phi\circ\mu_{i, h}(S^{n-1})$ contains at most $L$ cells. 
	Since $Z$ is a Cayley $n$-complex, we have that $\phi\circ\mu_{i, h}$ bounds a copy of $D^n$. 
	Thus, we can extend the map $\phi$ to the cell $D_{i, h}$ attached via $\mu_{i, h}$ by sending the attached cell to the copy of $D^n$ bounded by $\phi\circ\mu_{i, h}$. 
	The number of $n$-cells in the image of $D_{i, h}$ is bounded by $\delta(L)$ where $\delta$ is the $(n-1)$-st filling function of $Z$. 
\end{proof}

\begin{remark}\label{rem:cellular}
	By subdividing the $n$-cells of $X$, we may obtain a cellular map $X'\to Y$. 
\end{remark}

\begin{definition}
	We say that a map $f\colon H\to G$ is a {\em coarse embedding} if there are functions $\rho_-, \rho_+\colon \R\to \R$ such that $\lim_{x\to \infty}\rho_-(x) = \lim_{x\to \infty}\rho_+(x) = \infty$ and where the following inequality holds: 
	$$\rho_-(d(x, y)) \leq d(f(x), f(y)) \leq \rho_+(d(x, y)).$$
\end{definition}

\begin{proposition}\label{prop:coarseimplieslip}
	Let $f\colon H\to G$ be a coarse embedding. 
	Then $f$ is $\rho_+(1)$-Lipschitz. 
\end{proposition}
\begin{proof}
	Let $x, y$ be arbitrary elements of $H$. 
	There is a sequence of elements $x = z_0, z_1, \dots, z_{\ell} = y$ such that $z_i$ and $z_{i+1}$ are at distance 1 apart and $\ell = d(x, y)$. 
	Since $z_i$ and $z_{i+1}$ are at distance 1, we have that $d(f(z_i), f(z_{i+1}))\leq \rho_+(1)$ and so $d(f(x), f(y))\leq \sum_i d(f(z_i), f(z_{i+1})) \leq \ell \rho_+(1) = \rho_+(1)d(x, y). $
\end{proof}

In the presence of a coarse embedding, we may quantify how badly injectivity fails at the level of $n$-skeleta. 

\begin{lemma}\label{lem:coarse}
	Let $X$ be a cocompact Cayley $n$-complex for $H$. 
	Let $Z$ be a Cayley $n$-complex for $G$ with cocompact $n-1$ skeleton. 
	Let $Y$ and $\phi\colon X\to Y$ be the complex and map as in \cref{lem:injection}.
	Let $X'$ be the subdivision from \cref{rem:cellular}. 
	Suppose further that $f$ is a coarse embedding.

	Then there is a constant $L$ such that if $x_1, x_2\in X'$ are such that $f(x_1) = f(x_2)$, then $d(x_1, x_2)<L$. 
\end{lemma}
\begin{proof}
	Let $x_1, x_2$ be points of $X'$ that are identified in $Y$. 
	Since $\phi$ is an injection on $X^{(n-1)}\subset X'$, we see that $x_1, x_2$ must belong to the interior of an $n$-cell. 
	Let $N$ be the constant from \cref{lem:injection}. 
	We can see that there are vertices $v_1, v_2$ of $X$ such that $d(f(v_i), f(x_i))\leq N$. 
	Thus, $d(f(v_1), f(v_2))\leq 2N$. 
	Let $L' = \sup \{t\in \R\mid \rho_-(t)\leq 2N+1\}$.
	We can now see that $d(v_1, v_2)\leq L'$ and $d(x_1, x_2)\leq L' + 2N = L$
\end{proof}
In fact, the conclusions of this lemma do not require coarse embedding. See \cref{fig:notcoarse} for an example where the conclusions of this lemma hold but the map is very far form a coarse embedding. 

\begin{theorem}\label{thm:qi}
	Let $X, Y, \phi$ be as in \cref{lem:injection}. 
	Let $M = \phi(X)\subset Y$. 
	Give $M^{(1)}$ the graph metric and extend this to $M$. Then $X$ is quasi-isometric to $M$. 
\end{theorem}
\begin{proof}
	Let $X'$ be the complex obtained from \cref{rem:cellular}. 
	Since $\phi\colon X'\to Y$ is a cellular map, we see that it is Lipschitz when restricted to 1-skeleta which gives us the upper bound in the quasi-isometry. 
	
	To obtain the lower bound, let $x_1, x_2$ be two vertices in $X'$. 
	Let $y_i = \phi(x_i)$. 
	Let $d_1, \dots, d_{\lambda}$ be a minimal path from $y_1$ to $y_2$ in $M^{(1)}$. 
	Hence, we have $d(y_1, y_2) = \lambda$.
	Since $\phi\colon X'\to M$ is surjective and cellular, we can find an edge $e_i$ in $X'$ such that $\phi(e_i) = d_i$. 
	
	By \cref{lem:coarse}, we must have that $d(\tau(e_i), \iota(e_{i+1}))\leq L$.
	Also, $d(x_1, \iota(e_1)), d(\tau(e_{\lambda}), x_2)\leq L$. 
	Thus, we can find a path in $X^{(1)}$ from $x_1$ to $x_2$ with at most $L\lambda + \lambda + 2L$ edges. 
	That gives an upper bound for $d(x_1, x_2),$ and thus, we obtain the following inequality: 
	$$\frac{d(x_1, x_2)}{L+1} - \frac{2L}{L+1}\leq d(y_1, y_2). \qedhere$$
\end{proof}

The reader should have the following example in mind. 
Let $Y = \mathbb{H}^3$ and $X = \R^2$. 
Let $\phi\colon X\to Y$ be the inclusion of $X$ as a horosphere in $Y$. 
Then $\phi(X)$ and $M$ are isometric. However, $\phi$ is not a quasi-isometric embedding. Another example of a coarse embedding $\Z\to \Z^2$ is in \cref{fig:logmap}.

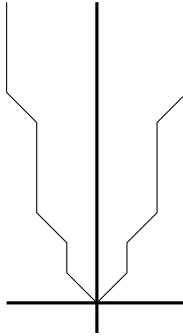
\begin{figure}[H]
	\centering
	\begin{tikzpicture}[scale=0.4]
	\draw (0,0) -- (1,1) -- (1,2) -- (2, 3) -- (2, 6) -- (3, 7) -- (3, 10);
	\draw (0,0) -- (-1,1) -- (-1,2) -- (-2, 3) -- (-2, 6) -- (-3, 7) -- (-3, 10);
	\draw [very thick] (0, -1) -- (0, 10);
	\draw [very thick] (-3, 0) -- (3, 0);
	\end{tikzpicture}
	\caption{The map $n\mapsto (\log_2(|n|+1), n)$ is a coarse embedding which is not a quasi-isometry. However, it is easy to see that the image of the induced map on Cayley graphs with the induced path metric is quasi-isometric to $\R$. }
	\label{fig:logmap}
\end{figure}

\section{Homological filling functions}\label{section:filling_functions}
Throughout this section, we denote $C_d(X, \mathbb{Z})$ as the cellular $d$-dimensional chain group $X$. Define a norm on $C_d(X, \Z)$ by $|\sum_{\sigma}n_{\sigma}\sigma| = \sum_{\sigma}|n_{\sigma}|$, where the sum runs over all $d$-cells of $X$. Note that this is well defined as these are finite sums. Finally, let $Z_d(X, \mathbb{Z})$ be the $\mathbb{Z}$-module of integral $d$-cycles of $X$,  and let $\partial \colon C_{d+1}(X, \mathbb{Z}) \to Z_d(X, \mathbb{Z})$ be the boundary map. Finally, we let $H_d(X, \mathbb{Z}) = \ker(\partial_d) / \text{Im}(\partial_{d+1})$ be the $d$-th integral homology group of $X$. 
\begin{definition}\label{def:sfill}
	Let $S$ be a collection of $(n-1)$-cycles in a space $X$. 
	Suppose that each element of $S$ is the boundary of an $n$-chain. 
	Define the {\em filling volume} of $s\in S$ to be:
	$$\fvol_X(s) = \min\left\{ \left. \sum_i |\alpha_i|  \right|  \sum_i \alpha_i \sigma_i \in C_n(X, \mathbb{Z}) \text{ where } \partial \left(\sum_i\alpha_i\sigma_i\right) = s\right\}.$$
	We define the {\em $S$-restricted $n$-th filling function of $S$} to be:
	$$\homfill_{X, S}^{n}(\ell) = \max \left\{\fvol_X(s)\mid |s|\leq \ell, s \in S \right\}.$$
\end{definition}

In the case that $S$ is the collection of all $(n-1)$-cycles in $X$, we obtain the usual homological filling function for $X$ which we denote as $\homfill_X^{n}(\ell)$. Also, for any collection $S$ of $(n-1)$-cycles, we have that $\homfill_{X, S}^n(\ell)\prec \homfill_{X}^n(\ell)$.
In the case that $X$ has a group action that is cocompact on $X^{(n-1)}$, then $\homfill_{X}^n(\ell)$ is well-defined. 

The following lemma follows from \cite{alonso, young1} where the action of a group was not used to show invariance under quasi-isometry. We produce it here for convenience. 

\begin{lemma}\label{lem:equivfill}
	Let $X, M, \phi$ be as in \cref{lem:injection}. 
	If $S$ be the set of all $(n-1)$-cycles in $X$,
	then $$\homfill_X^{n}(\ell) = \homfill_{X, S}^{n}(\ell)\sim \homfill_{M, S}^{n}(\ell).$$
\end{lemma}
\begin{proof}
	Let $X'$ be the complex from \cref{rem:cellular}.
	Each $n$-cell of $X$ is subdivided into at most $N$ $n$-cells of $X'$ where $N$ is the constant from \cref{lem:injection}. Thus, we obtain $\homfill_{X, S}^{n}(\ell)\sim \homfill_{X', S}^{n}(\ell)$. 
	
	Since we have a cellular map $X'\to M$, a filling in $X'$ maps to a filling in $M$, and thus, $\fvol_{X'}(s) \geq \fvol_M(s)$. 
	As such, we obtain the upper bound $\homfill_{M, S}^{n}(\ell)\prec \homfill_{X', S}^{n}(\ell)$. 
	
	To prove the other direction, we construct a quasi inverse to $\phi$, 
	i.e., a $(K, C)$-quasi-isometry $\psi$ such that $\psi(\phi(x))\leq C$ for all $x\in X$. 
	Let $g$ be the inverse of $\phi|_{X^{(n-1)}}$. 
	By construction, $\phi$ is injective on $X^{(n-1)}$. Therefore, $g$ is well-defined. 
	Now for each vertex of $M \backslash \phi(X^{(n-1)})$ define $\psi(v)$ to be any of the closest vertices of $\phi(X^{(n-1)})$. 
	Since $X$ is $(n-1)$-connected, we can extend this to higher cells by sending a cell to a minimal filling of its boundary. 
	
	Let $\delta^m$ be the $m$-th isoperimetric function for $X$. 
	Since at each stage of the process of \cref{lem:injection}, we were only taking finitely many cells for our fillings, we see that there exists a $k$ such that the boundary of each $\ell$-cell of $M$ contains $\leq k$ $(\ell-1)$-cells of $M$. 
	Thus, the boundary of each $n$-cell of $M$ is sent to at most $k(\delta^{n-2}(\dots k(\delta^0(K+C)\dots)$ $(n-1)$-cells of $X$. 
	Hence, each $n$-cell is sent to at most $\delta^{n-1}(k(\delta^{n-2}(\dots k(\delta^0(K+C)\dots)$ $n$-cells, and so, we see that fillings are changed by at most a constant multiple. Therefore, the functions are equivalent. 
\end{proof}

The above lemma allows us to study fillings in $M$ and obtain information about fillings in $X$. 
However, we need relate this to fillings in $Y$. 
In general, there will be more efficient fillings in $Y$ that are not contained in $M$.
For instance, in the previous example where $M$ is a horosphere in $\mathbb{H}^3$, the fillings in $M$ are quadratic in the boundary length but there are linear fillings in $\mathbb{H}^3$. 
The next lemma shows that this issue vanishes for top-dimensional filling functions. 

\begin{lemma}\label{lem:fillinginsub}
	Suppose that $Y$ is an $n$-dimensional complex with $H_n(Y, \mathbb{Z}) = 0$. 
	Let $M$ be a subspace and $c$ be an $(n-1)$-cycle in $M$. 
	Suppose that $c$ is the boundary of an $n$-cycle $d = \sum_i \alpha_i \sigma_i$ in $M$. 
	Then $\fvol_M(c) = \fvol_Y(c)$. 
\end{lemma}
\begin{proof}
	We show the stronger statement that homological fillings of $(n-1)$-cycles in $Y$ are unique. 
	
	Since $Y$ is $n$-dimensional and $H_n(Y, \mathbb{Z}) = 0$, we see that $\ker(\partial_n) = 0$. 
	Now suppose we had another filling $d'$ for $c$. 
	We then see that $d-d'$ is an element of $\ker(\partial_n)$ and hence is trivial. 
	Thus, $d = d'$, and the result follows. 
\end{proof}
We can now use this to bound the homological filling function of $M$ in terms of that of $Y$. 
\begin{corollary}\label{cor:fillingbound}
	Let $Y, M$ be as in \cref{lem:fillinginsub}. 
	Then $\homfill_M^{n-1}(\ell) \prec \homfill_Y^{n-1}(\ell)$. 
\end{corollary}

\section{Proof of Main Theorems}\label{section:main_results}
In this section, we will prove Theorems \ref{main_thm1}, \ref{main_thm2}, and \ref{main_thm3}. We start with the definition of the $n$-th homological filling function for a group of type $F_k$.
\begin{defn}
Let $G$ be a group acting properly, cocompactly, and by cellular automorphisms on a $k$-connected cell complex $X$. The $k$-th homological Dehn filling function of $G$ is the function $\homfill_G^{k}(\ell) \colon \mathbb{N} \to \mathbb{N}$ given by $\homfill_G^{k}(\ell) = \homfill_{X}^{k}(\ell)$.
\end{defn}

We remark that Young \cite{young1} proved that $\homfill_G^{k}(\ell)$ is a well defined invariant of a group meaning that if $G$ acts properly, cocompactly, and by cellular automorphisms on two $k$-connected cell complexes $X$ and $Y$, then $\homfill_X^{k}(\ell) \sim \homfill_Y^{k}(\ell)$.

We start with Theorem \ref{main_thm1} whose statement we recall for the convenience of the reader.

\begin{theorem}\label{main_thm1}
Let $G$ be a finitely presented group of geometric dimension $2$, and suppose that $H$ is a finitely presented group which admits a coarse embedding into $G$. Then $\homfill_{H}^2(\ell) \prec \homfill_G^{2}(\ell)$.
\end{theorem}
\begin{proof}
	We split into two cases, first assuming that $G$ admits a compact $K(G,1)$ of dimension $2$.
	Let $K$ be the universal cover of this $K(G, 1)$, i.e. $K$ is a contractible 2-dimensional cell-complex with a free, proper, and cocompact action of $G$. 
	Let $f \colon H \to G$ be the given coarse embedding, and let $X$ be the universal cover of the presentation complex of $H$ associated to some finite presentation of $H$. 
	By \cref{prop:coarseimplieslip} we have that $f$ is $C$-Lipschitz for some constant $C$. 
	By \cref{lem:injection}, there exists a Cayley $2$-complex $Y$ for $G$ that contains $K$ as a subcomplex which has a cocompact $2$-skeleton and a map $\phi \colon X \to Y$ which extends $f$ and is an injection on $1$-skeleta. Moreover, we have that each $2$-cell is mapped to a collection at most $N$ $2$-cells in $Y$, and by subdividing if necessary, we may assume that $\phi$ is a cellular map. Define $M = \phi(X)$. Letting $S$ be the set of $1$-cycles of $X$ in $Y$, \cref{lem:equivfill} implies that $\homfill_{M,S}^2(\ell) \sim \homfill_{X,S}^2(\ell)$. Since \cref{cor:fillingbound} implies that $\homfill_M^{2}(\ell) \prec \homfill_Y^1(\ell)$, we have that $\homfill_{X}^2(\ell) \prec \homfill_Y^{2}(\ell)$. Therefore, we have that $\homfill_H^2(\ell) \prec \homfill_G^2(\ell)$ as desired. 

	Now suppose that $G$ does not admit a cocompact $K(G,1)$ of dimension $2$.
	By contracting a maximal tree in $K(G, 1)^{(1)}$, we may assume that there is a single vertex in $K(G, 1)$. 
	Thus, this classifying space is the presentation 2-complex of an infinite aspherical presentation $\langle S\mid R\rangle$. 
	Let $K$ be the universal cover of this presentation 2-complex. 
	
	Since $G$ is finitely generated, there exists a finite subset of $S$ that generates $G$. 
	Let $\{x_1, \dots, x_\ell\} = A$ be such a subset.
	We have an inclusion $\psi$ of the Cayley graph $\Gamma$ of $G$ with respect to $A$ to the complex $K$. 
 We can find a finite presentation of the form $\langle A\mid w_1, \dots, w_m\rangle. $
	Let $Z$ be the universal cover of the presentation complex associated to this presentation. 
	
	We can extend $\psi$ to a map $\phi\colon Z\to K$ as follows. 
	Each relation $w_i$ is trivial in $H$ and thus bounds a disk in $H$. This disk is spanned by the relators in $R$. 
	We can now map the disk with boundary $w_i$ to the image of a disk diagram for $w_i$ in $K$. 
	We now extend this equivariantly to all of $Z$. 
	
	Let $M = \phi(Z)$. Since $\phi$ is $G$-equivariant, we see that the action of $G$ on $Z$ gives an action of $G$ on $M$. Since every point of $M$ is at bounded distance from a point in the image of $\Gamma$, we see that this action is proper and cocompact. Thus, $Z$ is quasi-isometric to $M$. 
	By \cref{lem:equivfill}, we have that $\homfill_{Z}^2(\ell) \sim \homfill_{M}^2(\ell)$. 
	
	Using \cref{lem:injection}, we can obtain a map $X\to Z$ where $X$ is the universal cover of a presentation complex for $H$. 
	We can compose with the map $Z\to M$. Let $L$ be the image of $X$ in $M$. \cref{lem:coarse} shows that $L$ and $X$ are quasi-isometric and so $\homfill_X^2(\ell)\sim\homfill_L^2(\ell)$. 
	
	Since $K$ is contractible, we have that $H_2(K, \mathbb{Z}) = 0$. 
	Since there are no 3-cells in $K$, we see that $\partial\colon C_2(K, \mathbb{Z})\to C_1(K, \mathbb{Z})$ is injective. 
	Thus, if we restrict to $C_2(L, \mathbb{Z})$, we also obtain an injection. 
	Therefore, $H_2(L, \mathbb{Z}) = 0$. 
	
	Now, by \cref{lem:fillinginsub}, we see that $\homfill_L^2(\ell)\prec\homfill_M^2(\ell)$. 	Thus, $\homfill_H^2(\ell)\prec\homfill_G^2(\ell)$.
\end{proof}

Theorem \ref{main_thm2} is a direct consequence of Theorem \ref{main_thm1}. However, we felt that Theorem \ref{main_thm2} is interesting in its own right and demonstrates that hyperbolicity of groups of geometric dimension $2$ passes to finitely presented groups which admit a coarse embedding to the given hyperbolic group. This is contrast to higher geometric dimensions where there exist hyperbolic groups which admit coarse embeddings of $\mathbb{Z}^2$.

\begin{theorem}\label{main_thm2}
	Let $G$ be a hyperbolic group of geometric dimension $2$, and suppose $H$ is a finitely presented group that admits a coarse embedding into $H$. Then $H$ is a hyperbolic group.
\end{theorem}
\begin{proof}
	Due to \cite[Theorem 5.2]{gersten}, we know that a finitely presented group $G$ is word hyperbolic if and only if $\homfill_{H}^2(\ell) \sim \ell$. Theorem \ref{main_thm1} implies that $\homfill_ H^{2}(\ell) \prec \homfill_{G}^2(\ell) \sim \ell$. Therefore, $\homfill_H^2(\ell) \sim \ell$ which implies that $H$ is hyperbolic as desired.
\end{proof}

The proof of Theorem \ref{main_thm3} is very similar to the proof of Theorem \ref{main_thm1} in the case where $G$ has a compact $K(G,1)$ of dimension $2$. However, we include its proof for completeness.

\begin{theorem}\label{main_thm3}
Let $k \geq 1$ Suppose that $G$ admits a finite $(k+1)$-dimensional $K(H,1)$ and that $H$ is a finitely presented group of type $F_{k+1}$. If $H$ admits a coarse embedding into $G$, then $\homfill_H^{k + 1}(\ell) \prec \homfill_G^{k + 1}(\ell)$. 
\end{theorem}
\begin{proof}
	 Let $f \colon H \to G$ be the given coarse embedding, and let $X$ be the universal cover of a $K(H,1)$ which a compact $(k+1)$-skeleta.  By definition, we have that $f$ is an injective $C$-Lipschitz map for some constant $C>0$. By Lemma \ref{lem:injection}, there exists a Cayley $k$-complex $Y$ for $G$ with cocompact $k$-skeleton and a map $\phi \colon X \to Y$ which extends $f$ and is an injection on $k$-skeleta. Moreover, we have that each $(k+1)$-cell is mapped to a collection at most $N$ $(k+1)$-cells in $Y$, and by subdividing if necessary, we may assume that $\phi$ is a cellular map. Define $M = \phi(X)$. Letting $S$ be the set of $k$-cycles of $X$ in $Y$, Lemma \ref{lem:equivfill} implies that $\homfill_{M,S}^{k+1}(\ell) \sim \homfill_{X,S}^{k+1}(\ell)$. Since Corollary \ref{cor:fillingbound} implies that $\homfill_M^{k}(\ell) \prec \homfill_Y^{k+1}(\ell)$, we have that $\homfill_{X}^{k+1}(\ell) \prec \homfill_Y^{k + 1}(\ell)$. Therefore, we have $\homfill_H^{k + 1}(\ell) \prec \homfill_G^{k+1}(\ell)$ as desired. 
\end{proof}

\section{Small cancellation groups}\label{small_cancellation}

In this section, we prove \cref{thm:smallcancellation}. For the following, we let $F(A)$ be the free group on the set $A$.

\begin{definition}
	Let $A$ be a finite set. Let $R$ be a subset of $F(A)$. We say that $p\in F(A)$ is a {\em piece} in $R$ if there exists cyclic permutations $s_i \neq s_j$ of relators $r_i, r_j\in R$ such that $s_i = pu$ and $s_j = pv$. 
	
	We say that a presentation $\langle A\mid R\rangle$ satisfies the {\em small cancellation condition $C'(\lambda)$ } if given a piece $p$ in a relator $r$, we have that $|p|<\lambda |r|$. 
\end{definition}

We refer the reader to \cite[Section V]{lyndonschupp} for a full treatment of small cancellation theory. 
The following will be key to our theorem. 

\begin{theorem}
	If $G = \langle A\mid R\rangle $ satisfies $C'(\frac{1}{6})$ and no relator is a proper power, then
	\begin{itemize} 
		\item The presentation 2-complex of $G$ is aspherical, 
		\item $G$ has geometric dimension 2. 
		\item If $R$ is finite, then $G$ is hyperbolic. 
	\end{itemize}
\end{theorem}

The first two points follow from \cite{chiswell, olshanskii}. The latter follows from the following theorem which shows that $G$ has linear Dehn function. 
\begin{theorem}\cite[Theorem V.4.5]{lyndonschupp}\label{green}
	Let $G = \langle A\mid R\rangle$ be a small cancellation group where $R$ is finite. Let $w$ be a word in $A$ which represents the trivial element in $G$.
	Then $w$ contains a sub-word which is more than half a relator. 
\end{theorem}

We can in fact prove this theorem without the requirement that $R$ is finite. 

\begin{lemma}
	Suppose $\langle A\mid R\rangle $ satisfies $C'(\frac{1}{6})$ with $A$ finite. 
	Suppose that $w$ is a word representing the identity. 
	Then $w$ contains a sub-word which is more than half a relator. 
\end{lemma}
\begin{proof}
	Since $w$ is trivial, we can write it as a product of conjugates of relators.
	This is a finite product, and thus, we let $r_1, \dots, r_l$ be the relators appearing in this product. 
	Then $w$ is trivial in the group $\langle A\mid r_1, \dots, r_l\rangle$. 
	Since $C'(\frac{1}{6})$ is preserved under taking sub presentations, \cref{green} shows that $w$ contains a sub-word which is more than half of one of the relators $r_1, \dots, r_l$. 
\end{proof}

This allows us to understand the filling function of the universal cover of the presentation complex. 

\begin{corollary}\label{cor:linearinfinitepres}
	Let $\langle A\mid R\rangle$ be a $C'(\frac{1}{6})$ small cancellation group. Let $Z$ be the universal cover of the presentation complex. 
	Then $\homfill_{Z}^2(\ell) \sim \ell$. 
\end{corollary}

We are now ready to prove our theorem. 

\begin{theorem}\label{thm:smallcancellation}
	Let $G = \langle A\mid R\rangle$ be a $C'(\frac{1}{6})$ presentation where no relator is a proper power. If $H$ is a finitely presented group that admits a coarse embedding into $G$, then $H$ is hyperbolic. In particular, all finitely presented subgroups of $G$ are hyperbolic.
\end{theorem}
\begin{proof}
	Let $H$ be a finitely presented group with a coarse embedding $f\colon H\to G$. 
	Let $Z$ be the universal cover of the presentation complex for $G$. 
	Let $X$ be the universal cover of a presentation complex for $H$. 
	By \cref{lem:injection}, we have a cellular map $\phi\colon X\to Y$ where $Y$ is a Cayley 2-complex for $H$ containing $Z$ as a deformation retract. 
	Thus, $Y$ is contractible. 
	
	Let $M = \phi(X)$. Then \cref{thm:qi} shows that $X$ and $M$ are quasi-isometric, and by \cref{lem:equivfill}, they have equivalent filling functions. 
	
	Since $Y$ is aspherical, we see by \cref{lem:fillinginsub} that $\homfill_{M}^2(\ell)\prec \homfill_Y^2(\ell)$. The latter is equivalent to a linear function by \cref{cor:linearinfinitepres}. Thus, $X$ satisfies a linear filling function, and hence, $H$ is a hyperbolic group by \cite[Theorem 3.1]{gersten}. 
\end{proof}

\section{Future Questions}\label{section:future}
We finish with some discussion of questions of interest to the authors. Throughout this article, we considered coarse embeddings. However, one could weaken this condition by considering Lipschitz embeddings which are not coarse embeddings. In many cases, our theorems still hold. For example, in \cref{fig:notcoarse}, we show an example where $f$ is a Lipschitz embedding which is not a coarse embedding but the subcomplex is quasi-isometric to the group. This gives evidence to a possible strengthening of Theorem \ref{main_thm1} and \ref{main_thm3} where Lipschitz embeddings take the place of coarse embeddings. We do not know of a counterexample at this time. Therefore, we have the following questions.

\begin{figure}[H]
	\centering
	\begin{tikzpicture}[scale=0.4]
	\draw (-1, -1) -- (1, -1) -- (1, 1) -- (-3, 1) -- (-3, -5) -- (5, -5) -- (5, 3);
	\draw (-1, -1) -- (-1, -3) -- (3, -3) -- (3, 3) -- (-5, 3) -- (-5, -7) -- (7, -7) -- (7, 3);
	\draw [thick] (-6, -2) -- (8, -2);
	\draw [thick] (0, -8) -- (0, 4);		
	\end{tikzpicture}
	\caption{An example of a Lipschitz injection $\Z\to \Z^2$. In this example for any $D$ there are points in $\Z$ of distance $\geq D$ such that there images have distance $1$. However, the induced path metric on the subcomplex is quasi-isometric to $\Z$. }
	\label{fig:notcoarse}
\end{figure}
\begin{ques}\label{q1}
Suppose that $H$ is a group of type $F_{k+1}$. Suppose that $G$ admits a finite $(k+1)$-dimensional $K(G,1)$. If $H$ admits a Lipschitz embedding into $G$, does there exist an inequality of $(k+1)$-th dimensional homological Dehn functions? 

If not, can one find an example of groups $G$ and $H$ as above where there exist a Lipschitz embedding of $H$ into $G$ where $H$ has strictly higher $(k+1)$-th dimensional homological Dehn function than $G$?
\end{ques}

Through Theorem \ref{main_thm2}, we generalize \cite[Theorem 5.4]{gersten} to coarse embeddings. However, one may ask if the same holds for a general Lipschitz embedding.
\begin{ques}
Let $G$ be a hyperbolic group of geometric dimension $2$. If $H$ admits a Lipschitz embedding into $G$, is $H$ necessarily hyperbolic?
\end{ques}

Another finiteness question that we are interested in involves cohomological and homological dimension over a ring $R$. As originally proved by Sauer \cite{sauer}, if we are given a coarse embedding of $H$ into $G$ where both $G$ and $H$ have finite cohomological (homological) dimension over $R$, then the cohomological (homological) dimension of $H$ is bounded by that of $G$. We may ask a similar question for Lipschitz embeddings.
\begin{ques}
Suppose that $H$ admits a Lipschitz embedding into $G$ and that $G$ and $H$ both have finite cohomological (homological) dimension over a commutative ring $R$. Is the $R$-cohomological ($R$-homological) dimension of $H$ bounded above by that of $G$? 

If not, can one find an example of a Lipschitz embedding of groups $G$ and $H$ where $H$ has strictly greater $R$-cohomological dimension than that of $G$?
\end{ques}

\bibliographystyle{plain}
\bibliography{bibs}

\begin{thebibliography}{10}

\bibitem{alonso}
J.~M. Alonso, X.~Wang, and S.~J. Pride.
\newblock Higher-dimensional isoperimetric (or {D}ehn) functions of groups.
\newblock {\em J. Group Theory}, 2(1):81--112, 1999.

\bibitem{chiswell}
Ian~M. Chiswell, Donald~J. Collins, and Johannes Huebschmann.
\newblock Aspherical group presentations.
\newblock {\em Math. Z.}, 178(1):1--36, 1981.

\bibitem{Cohen_Pengitore}
David~Bruce Cohen and Mark Pengitore.
\newblock Translation-like actions of nilpotent groups.
\newblock {\em J. Topol. Anal.}, 11(2):357--370, 2019.

\bibitem{epstein}
David B.~A. Epstein, James~W. Cannon, Derek~F. Holt, Silvio V.~F. Levy,
  Michael~S. Paterson, and William~P. Thurston.
\newblock {\em Word processing in groups}.
\newblock Jones and Bartlett Publishers, Boston, MA, 1992.

\bibitem{gersten}
S.~M. Gersten.
\newblock Subgroups of word hyperbolic groups in dimension {$2$}.
\newblock {\em J. London Math. Soc. (2)}, 54(2):261--283, 1996.

\bibitem{gruber}
Moritz Gruber.
\newblock Filling invariants of stratified nilpotent {L}ie groups.
\newblock {\em Math. Z.}, 293(1-2):39--79, 2019.

\bibitem{hanlon_pedroza}
Richard~Gaelan Hanlon and Eduardo Mart\'{\i}nez~Pedroza.
\newblock A subgroup theorem for homological filling functions.
\newblock {\em Groups Geom. Dyn.}, 10(3):867--883, 2016.

\bibitem{hume_sisto}
David Hume and Alessandro Sisto.
\newblock Groups with no coarse embeddings into hyperbolic groups.
\newblock {\em New York J. Math.}, 23:1657--1670, 2017.

\bibitem{kropholler_kielak}
Dawid Kielak and Peter Kropholler.
\newblock Isoperimetric inequalities for poincar\'{e} duality groups.
\newblock {\em arXiv preprint arXiv:2008.07812}, 2020.

\bibitem{lang}
Urs Lang.
\newblock Higher-dimensional linear isoperimetric inequalities in hyperbolic
  groups.
\newblock {\em Internat. Math. Res. Notices}, (13):709--717, 2000.

\bibitem{lyndonschupp}
Roger~C. Lyndon and Paul~E. Schupp.
\newblock {\em Combinatorial group theory}.
\newblock Classics in Mathematics. Springer-Verlag, Berlin, 2001.
\newblock Reprint of the 1977 edition.

\bibitem{mineyev}
Igor Mineyev.
\newblock Bounded cohomology characterizes hyperbolic groups.
\newblock {\em Q. J. Math.}, 53(1):59--73, 2002.

\bibitem{olshanskii}
A.~Yu. Ol'shanski\u{\i}.
\newblock {\em Geometry of defining relations in groups}, volume~70 of {\em
  Mathematics and its Applications (Soviet Series)}.
\newblock Kluwer Academic Publishers Group, Dordrecht, 1991.
\newblock Translated from the 1989 Russian original by Yu. A. Bakhturin.

\bibitem{sauer}
R.~Sauer.
\newblock Homological invariants and quasi-isometry.
\newblock {\em Geom. Funct. Anal.}, 16(2):476--515, 2006.

\bibitem{varopoulos_saloff-coste_coulhon}
N.~Th. Varopoulos, L.~Saloff-Coste, and T.~Coulhon.
\newblock {\em Analysis and geometry on groups}, volume 100 of {\em Cambridge
  Tracts in Mathematics}.
\newblock Cambridge University Press, Cambridge, 1992.

\bibitem{young1}
Robert Young.
\newblock Homological and homotopical higher-order filling functions.
\newblock {\em Groups Geom. Dyn.}, 5(3):683--690, 2011.

\bibitem{younglinear}
Robert Young.
\newblock The {D}ehn function of {${\rm SL}(n;\Bbb Z)$}.
\newblock {\em Ann. of Math. (2)}, 177(3):969--1027, 2013.

\end{thebibliography}
\end{document}